\documentclass[11pt]{amsart}
\usepackage{amsmath, amssymb, amsthm, bm, mathtools, enumitem, caption}
\usepackage{hyperref}

\usepackage[noabbrev,capitalize]{cleveref}
\usepackage{adjustbox}
\crefname{equation}{}{}
\usepackage{ytableau}
\usepackage{graphicx, tikz}
\usepackage[textheight=8.5in, textwidth=6.75in]{geometry}

\usepackage{microtype}

\newtheorem{theorem}{Theorem}[section]
\newtheorem{lemma}[theorem]{Lemma}
\newtheorem{corollary}[theorem]{Corollary}
\newtheorem{proposition}[theorem]{Proposition}

\newtheorem*{conjecture*}{Conjecture}

\theoremstyle{definition}

\newtheorem*{definition}{Definition}

\theoremstyle{remark}
\newtheorem*{remark}{Remark}

\newtheorem*{TwoRemarks}{Two Remarks}
\newtheorem*{example}{Example}
\newtheorem*{examples}{Examples}

\numberwithin{equation}{section}

\DeclareMathOperator{\Tr}{Tr}

\DeclareMathOperator{\lcm}{lcm}
\DeclareMathOperator{\denom}{denom}

\DeclareMathOperator{\Div}{Div}

\newcommand{\C}{\mathbb C}
\newcommand{\SL}{\mathrm{SL}}

\newcommand{\Q}{\mathbb Q}

\newcommand{\Z}{\mathbb Z}
\newcommand{\ord}{\mathrm{ord}}

\title[Traces of Partition Eisenstein series]{Traces of Partition Eisenstein series}
\author{Tewodros Amdeberhan, Michael Griffin, Ken Ono, \and Ajit Singh}
\address{Dept. of Mathematics, Tulane University, New Orleans, LA 70118}
\email{tamdeber@tulane.edu}
\address{Dept. of Mathematics, Vanderbilt University, Nashville, TN 37240}
\email{michael.j.griffin@vanderbilt.edu}
\address{Dept. of Mathematics, University of Virginia, Charlottesville, VA 22904}
\email{ko5wk@virginia.edu}
\email{ajit18@iitg.ac.in}

\begin{document}
\begin{abstract} 
We study ``partition Eisenstein series'', extensions of the Eisenstein series $G_{2k}(\tau),$ defined by
$$\lambda=(1^{m_1}, 2^{m_2},\dots, k^{m_k}) \vdash k \ \ \ \ \ \longmapsto \ \ \ \ \ 
G_{\lambda}(\tau):= G_2(\tau)^{m_1} G_4(\tau)^{m_2}\cdots G_{2k}(\tau)^{m_k}.
$$
For functions $\phi  :  \mathcal{P}\mapsto \C$ on partitions,  the {\it weight $2k$ partition Eisenstein trace} is the quasimodular form
$$
\Tr_k(\phi;\tau):=\sum_{\lambda \vdash k} \phi(\lambda)G_{\lambda}(\tau).
$$
These traces give explicit formulas for some well-known generating functions, such as the $k$th elementary symmetric functions of the inverse points of 2-dimensional complex lattices $\Z\oplus \Z\tau,$ as well as the $2k$th power moments of the Andrews-Garvan crank function. To underscore the ubiquity of such traces, we show that their generalizations give the Taylor coefficients of generic Jacobi forms with torsional divisor.
 \end{abstract}

\maketitle

\section{Introduction and Statement of Results}

For positive integers $k$ and $\Im(\tau)>0,$ the weight $2k$ Eisenstein series (see Ch. 1 of \cite{CBMS}) is
\begin{equation}\label{Eisenstein}
G_{2k}(\tau) := -\frac{B_{2k}}{2k}+2\sum_{n}\sigma_{2k-1}(n)q^n \ = \ \frac{(2k-1)!}{(2\pi i)^{2k}}\sum_{\substack{\omega \in \Z\oplus \Z\tau\\
\omega\neq 0}}\frac{1}{\omega^{2k}},
\end{equation}
where $B_{2k}$ is the  $2k$-th Bernoulli number, $\sigma_{\nu}(n):=\sum_{d\mid n}d^{\nu},$ and $q:=e^{2\pi i \tau}.$ The first few examples are
$$
G_2(\tau)=-\frac{1}{12}+2\sum_{n=1}^{\infty}\sigma_1(n)q^n,\ \ \
G_4(\tau)=\frac{1}{120}+2\sum_{n=1}^{\infty} \sigma_3(n)q^n,\ \ \
G_6(\tau)= -\frac{1}{252}+2\sum_{n=1}^{\infty}\sigma_5(n)q^n.
$$
Apart from $G_2,$ each $G_{2k}$ is a weight $2k$ holomorphic modular form on $\SL_2(\Z).$  
Quasimodular forms  are the functions in the polynomial ring
  (for example, see \cite{Zagier})
$$
\mathbb{C}[G_2,G_4,G_6]=\mathbb{C}[G_2, G_4, G_6, G_8, G_{10}, \dots].
$$

We study ``partition Eisenstein series'', which were introduced in \cite{AOS2}. These quasimodular forms are extensions of the classical Eisenstein series. 
To make this precise, recall that a {\it partition of  a non-negative integer $k$} (see \cite{Andrews} for background on partitions) is any nonincreasing sequence of positive integers 
$\lambda=(\lambda_1,\lambda_2,\dots, \lambda_s)$ that sum to $k$, denoted $\lambda\vdash k.$ Equivalently, we use the notation $\lambda=(1^{m_1},\dots,k^{m_k})\vdash k$, where $m_j$ is the multiplicity of $j.$ Furthermore, the {\it length} of $\lambda$ is $\ell(\lambda):=m_1+\dots+m_k.$
For such partitions, we
 define the weight $2k$ {\it partition Eisenstein series}\footnote{These $G_{\lambda}$ should not be mistaken for the partition Eisenstein series introduced by Just and Schneider \cite{JustSchneider}, which are semi-modular instead of quasimodular. We also note that these functions are scalar multiples of those introduced in \cite{AOS2}.}
\begin{equation}\label{primary}
\lambda=(1^{m_1}, 2^{m_2},\dots, k^{m_k}) \vdash k \ \ \ \ \ \longmapsto \ \ \ \ \  G_{\lambda}(\tau):= G_2(\tau)^{m_1}G_4(\tau)^{m_2}\cdots G_{2k}(\tau)^{m_k}.
\end{equation}
The Eisenstein series $G_{2k}(\tau)$ corresponds to the partition $\lambda=(k)$, as we have $G_{(k^1)}(\tau)=G_{2k}(\tau)^1.$

To define partition Eisenstein traces, suppose that $\phi: \mathcal{P}\mapsto \C$ is a function on partitions. For each positive integer $k$, its {\it partition Eisenstein trace}
is the weight $2k$ quasimodular form 
\begin{equation}\label{PartitionTrace}
\Tr_k(\phi;\tau):=\sum_{\lambda \vdash k} \phi(\lambda)G_{\lambda}(\tau).
\end{equation}
By convention, for $k=0$, we let $\Tr_0(\phi;\tau):=1$.

\begin{remark}
Such traces arise in recent work on MacMahon's sums-of-divisors $q$-series (see Theorem~1.4 of \cite{AOS}) and in an explicit proof of a claim from Ramanujan's ``lost notebook'' on $q$-series assembled from derivatives of theta functions (see Theorem~1.2 of \cite{AOS2}).
\end{remark}

We show that such traces give the elementary symmetric functions for the inverse points of 2-dimensional complex lattices
 $\Lambda_{\tau}:=\Z\oplus \Z\tau,$ where $\Im(\tau)>0.$
For each  $s\geq 0,$  we let
\begin{equation}
\Lambda_{\tau}(s):= \{\omega^2|\omega|^{2s}\ : \ \omega \in \Lambda_{\tau}\setminus \{0\}\}.
\end{equation}
In analogy with the classical elementary symmetric functions, for each  $k\geq 1,$ we define
\begin{equation}\label{symmetricsums}
e_{k}(\Lambda_{\tau}(s)):=\sum_{\substack{\omega_1,\dots, \omega_k \in \Lambda_{\tau}(s)\\ {\text{distinct}}}} \frac{1}{\omega_1 \cdots \omega_k}.
\end{equation}
For $\lambda=(1^{m_1}, \dots, k^{m_k})\vdash k,$ we define the function
\begin{equation}\label{LatticeFunction}
\phi_{\Lambda} ( \lambda):= {\color{black}\frac{(-1)^{\ell(\lambda)}}{\prod_{j=1}^k m_j! ((2j)!)^{m_j}}.}
\end{equation}

\begin{theorem}\label{Theorem1}
If $k$ is a positive integer, then the following are true.

\noindent
(1) We have that
$$e_{k}(\Lambda_{\tau}(0))=\displaystyle \lim_{s\to 0^+} e_{k}(\Lambda_{\tau}(s))
$$
is a weight $2k$ nonholomorphic modular form on $\SL_2(\Z).$

\noindent
(2) There is a degree $k$ isobaric polynomial\footnote{A polynomial is isobaric of degree $k$ if each of its monomials $X_1^{d_1}X_2^{d_2}\cdots X_k^{d_k}$ satisfies $d_1+2d_2+\dots+kd_k=k.$} $F_k(X_1, X_2,\dots, X_k)$ for which
$$e_{k}(\Lambda_{\tau}(0))={\color{black}(2\pi)^{2k}}F_k(G_2^{\star}(\tau), G_4(\tau),\dots, G_{2k}(\tau)),
$$
where $G_{2}^{\star}(\tau)$ is the nonholomorphic weight 2 modular form
$$
  G_2^{\star}(\tau)=\frac{1}{4\pi \Im(\tau)}+G_2(\tau).
$$  

\noindent
(3) The holomorphic part of
 $e_k(\Lambda_{\tau}(0))$ is given by
 
\[
F_k(G_2(\tau), G_4(\tau),\dots, G_{2k}(\tau)) = \Tr_k(\phi_{\Lambda};\tau).
\]
\end{theorem}

\begin{remark}
The sum in (\ref{symmetricsums}) converges absolutely 
for $s>0$, but converges only conditionally when $s=0$. In Theorem \ref{Theorem1}, this parameter is employed using ``Hecke's trick'' (for example, see p. 84 of \cite{BFOR}) which gives the convergence of $G_2^{\star}(\tau).$  
\end{remark}

Traces of partition Eisenstein series also give the first explicit formulas for the generating functions of the even power moments of the celebrated Andrews-Garvan crank function \cite{AG, Garvan}, whose discovery confirmed a speculation of Dyson \cite{Dyson} 
on Ramanujan's partition congruences.
For a partition $\lambda,$ let $l(\lambda)$ be the largest part, and let $\omega(\lambda)=m_1$ the number of $1's$ in $\lambda.$ If $\mu(\lambda)$ denotes the number of parts of $\lambda$ larger than $\omega(\lambda),$ then the {\it crank} $c(\lambda)$ is defined by
$$
c(\lambda):=\begin{cases} l(\lambda) \ \ \ \ \ &{\text {\rm if}}\ \omega(\lambda)=0,\\
\mu(\lambda)-\omega(\lambda) \ \ \ \ \ &{\text {\rm if}}\ \omega(\lambda)>0.
\end{cases}
$$
For every integer $a$, Andrews and Garvan proved the remarkable identities
\begin{displaymath}
\begin{split}
 \# \{ \lambda \vdash 5n+4 \ : \ c(\lambda)\equiv a\! \! \! \! \pmod 5\}&=\frac{p(5n+4)}{5},\\
 \# \{ \lambda \vdash 7n+5 \ : \ c(\lambda)\equiv  a\! \! \! \! \pmod 7\}&=\frac{p(7n+5)}{7},\\
 \# \{ \lambda \vdash 11n+6 \ : \ c(\lambda)\equiv a\! \! \! \!\pmod{11}\}&=\frac{p(11n+6)}{11},\\
\end{split}
\end{displaymath}
providing a combinatorial explanation for Ramanujan's partition congruences
\begin{displaymath}
\begin{split}
p(5n+4)&\equiv 0\pmod 5,\\ 
p(7n+5)&\equiv 0\pmod 7,\\
p(11n+6)&\equiv 0\pmod{11}.
\end{split}
\end{displaymath}

In an important paper, Atkin and Garvan \cite{AtkinGarvan} discovered infinite families of relations between this crank and Dyson's rank function $r(\lambda):=l(\lambda)-\ell(\lambda).$ Central to their work are the even power moments of $M(m,n),$ the number of partitions of $n$ with crank $m.$
Namely, for positive integers $k,$ they defined
\begin{equation} \label{crankGen}
 C_{2k}(q):=\sum_{n\geq0}\sum_{m\in\mathbb{Z}}m^{2k}M(m,n)\,q^n.
 \end{equation}
 They found a recurrence relation for these generating functions, which allows them to conclude, for each $k$, that
 $$
 C_{2k}(q)=\frac{1}{(q;q)_{\infty}}\cdot F_{2k}(\tau),
 $$
 where
 $(q;q)_{\infty}:=\prod_{n=1}^{\infty}(1-q^n)$
 and $F_{2k}(\tau)$ is a mixed weight quasimodular form. Here we obtain explicit formulas for these even power moments; they arise from traces of partition Eisenstein series.
 
\begin{theorem}\label{Theorem2}
The generating function for the even power crank moments is
\begin{align*}
	\sum_{k\geq0}(-1)^kC_{2k}(q)\cdot\frac{X^{2k}}{(2k)!}
		=\frac{2\sin(\frac{X}{2})}{(q;q)_{\infty}}\cdot
		\sum_{k\geq0} (-1)^k\Tr_k(\phi_c;\tau)\cdot X^{2k-1},
\end{align*}
where 
$$
\phi_c((1^{m_1}, 2^{m_2},\dots, k^{m_k})):=\frac{1}{\prod_{j=1}^k m_j!\left((2j)!\right)^{m_j}}.
$$
\end{theorem}

As a
consequence, we obtain explicit formulas for the even crank moments.

\begin{corollary}\label{Cor1} If $k$ is a positive integer, then we have
$$
	C_{2k}(q)=\frac{1}{(q;q)_\infty}\cdot\sum_{n=0}^{k} \frac{(2k)_{2k-2n}}{4^{n}\cdot(2n+1)}\cdot \Tr_{k-n}(\phi_{c};\tau),
$$	
where $(x)_m:=x(x-1)\cdots (x-m+1).$
\end{corollary}

\begin{examples} Here we offer two examples of Corollary~\ref{Cor1}. \newline \noindent
			(1) For $k= 3$, Corollary~\ref{Cor1} gives
			$$C_6(q)=\frac{1}{(q;q)_\infty}\left(\frac{1}{448}\Tr_0(\phi_c;\tau)+\frac{3}{8}\Tr_{1}(\phi_{c};\tau)+30\Tr_{2}(\phi_{c};\tau)+720\Tr_{3}(\phi_{c};\tau)\right).$$
			(2) For $k= 4$, Corollary~\ref{Cor1} gives
			$$C_8(q)=\frac{1}{(q;q)_\infty}\left(\frac{1}{2304}\Tr_0(\phi_c;\tau)+\frac{1}{8}\Tr_{1}(\phi_{c};\tau)+21 \Tr_{2}(\phi_{c};\tau)+1680\Tr_{3}(\phi_{c};\tau)+40320\Tr_{4}(\phi_{c};\tau)\right).$$
\end{examples}

\begin{remark}~ Atkin and Garvan expressed $C_2(q), C_4(q), C_6(q),$ and $C_8(q)$ (see (4.7) of \cite{AtkinGarvan})  in terms of the Lambert series
\begin{equation}\label{Sj} 
	\mathbf{S}_{2k-1}(q):=\sum_{m\geq1}\frac{m^{2k-1}q^m}{1-q^m} 
	= \frac{B_{2k}}{4k} +\frac{1}{2}G_{2k}(\tau). 
\end{equation}
Furthermore, they use induction  (see (4.8) of \cite{AtkinGarvan}) to show that each $C_{2k}(q)$ is a polynomial in such Lambert series. 
Corollary~\ref{Cor1} gives explicit formulas for these polynomial representations. A straightforward modification, where one uses (\ref{Sj}), yields the explicit formula
$$C_{2k}(q)
=\frac{(2k)!}{(q;q)_{\infty}} \cdot \sum_{\lambda\vdash k} \prod_{j=1}^k \frac1{m_j!} \left(\frac2{(2j)!}\right)^{m_j}\cdot \mathbf{S}_{\lambda}(q),$$
where for a partition $\lambda=(1^{m_1}, 2^{m_2},\dots, k^{m_k})\vdash k$ we let
$$
\mathbf{S}_{\lambda}(q):=\prod_{j=1}^{k} \mathbf{S}_{2j-1}(q)^{m_j}.
$$
Furthermore, we note that Rhoades (see Theorem 2.1 of  \cite{Rhoades}) proved formulas which are equivalent to Corollary 1.3.
\end{remark}

Theorems~\ref{Theorem1} and \ref{Theorem2} both arise naturally in the context of Jacobi forms. Therefore, it is natural to ask whether traces of partition Eisenstein series arise generally in this wider context. We prove that this is indeed the case.  Partition traces of more general Eisenstein series give Taylor coefficients of generic Jacobi forms (for background, see \cite{BFOR, EichlerZagier}) with ``torsional divisor.''
To  make this precise, we briefly recall some terminology. 

\begin{definition}A holomorphic function $F(z;\tau)$ on $\mathbb{C}\times \mathbb{H}$ is a  {\it Jacobi form for $\SL_2(\Z)$ of weight $k$ and index $m$} if it 
satisfies the following conditions:

\noindent
(1) For all $\gamma=\left(\begin{smallmatrix}a&b\\c&d\end{smallmatrix}\right)\in \SL_2(\Z),$ we have the modular transformation
\[F\left(\frac{z}{c\tau+d};\frac{a\tau+b}{c\tau+d}\right) \ = \ (c\tau+d)^k\exp\left(2\pi i\cdot \frac{mcz^2}{c\tau+d}\right) F(z;\tau).\]

\noindent
(2)  For all integers $a,b$, we have the elliptic transformation
\[
F(z+a\tau+b;\tau)  \ = \  \exp\big(-2\pi i m(a^2 \tau+2az)\big)F(z;\tau).
\]

\noindent
(3) The 
Fourier expansion of $F(z;\tau)$ is given by
\[F(z;\tau)=\sum_{n\geq 0}\sum_{r^2\leq 4mn} b(n,r)q^nu^r,\]
where $b(n,r)$ are complex numbers and $u:=e^{2\pi i z}$.
\end{definition}

\begin{TwoRemarks}~

\noindent
(1) To define Jacobi forms on congruence subgroups $\Gamma\subseteq \SL_2(\Z)$, one modifies (3) by requiring that $F$ has a similar Fourier expansion at each of the cusps of $\Gamma.$ 

\noindent
(2) From the definition of a Jacobi form, it is clear that if two Jacobi forms are multiplied, their weights and indexes are added. Meromorphic Jacobi forms of index $0$ are elliptic functions. 
\end{TwoRemarks}

We consider \emph{meromorphic Jacobi forms} (for example, see Chapter 11 of \cite{BFOR}), which are mild extensions of the holomorphic Jacobi forms defined above, where poles are permitted in $z$, and where the expansions in (3) allow for finite order poles in $\tau$ at cusps. The {\it divisor} of such a form is the formal sum 
$$\Div(F)=\sum_{x\in \C/\Lambda_\tau} a_x\cdot (x),
$$
where each $a_x$ is the order of $F$ at $x$. We note that divisors only have a finite number of nonvanishing summands. Furthermore, we say that a point $x=a\tau+b\in \C/\Lambda_\tau$ is a {\it torsion point} if  $a$ and $b$ are rational.
We prove that if  $\Div(F)$ is torsional (i.e. supported at torsion points), then the Taylor coefficients of $F$ with respect to $z$ are given as traces of partition Eisenstein series. 

To define this generalization, we require higher level Eisenstein series.  Specifically, in Section \ref{SecJacobi} we define an Eisenstein series $G_{k,D}(\tau)$, for each positive integer $k$ and formal divisor $D$ on $\C/\Lambda_\tau$. If $D=\Div(F),$  then these Eisenstein series have the same (or lower) level as $F$. 
If $\lambda=(1^{m_1},\dots,k^{m_k})\vdash k,$ then in analogy with (\ref{primary})  we define the {\it partition Eisenstein series for the divisor $D$} by 
\begin{equation}\label{highereisenstein}
G_{D,\lambda}(\tau):=G_{1,D}^{m_1}(\tau)G_{2,D}^{m_2}(\tau)\dots G_{k,D}(\tau)^{m_k}.
\end{equation}
We note that this notation allows for odd weight Eisenstein series, which might exist depending on $D$. Hence $G_{D, \lambda}(\tau)$ is a weight $k$ form, rather than a weight $2k$ form in the case of (\ref{primary}).
In turn, we define the {\it partition Eisenstein trace for the divisor $D$} by
\begin{equation}\label{PartitionTraceDiv}
\Tr_{k}(D,\phi;\tau):=\sum_{\lambda \vdash k} \phi(\lambda)G_{D, \lambda}(\tau).
\end{equation}

We now present our result about
 meromorphic Jacobi forms with torsional divisor.

\begin{theorem}\label{Theorem3}
Suppose $F(z;\tau)$ is a meromorphic Jacobi form of weight $k$ and index $m$, and torsional divisor $D=\Div(F)$, and let $a$ be the order of $F(z;\tau)$ at $z=0$. Then there is a unique weakly holomorphic modular form $f_F(\tau)$ of weight $k+a$ for which
\[
F(z;\tau) \ = \ f_F(\tau) \cdot 
\sum_{t\geq 0} \Tr_{t}(D,\phi_{J};\tau)\cdot {\color{black}(2\pi i z)}^{t+a},
\]
where for partitions $\lambda=(1^{m_1}, 2^{m_2},\dots, t^{m_t})\vdash t$ we let
\[
\phi_{J}(\lambda) := \frac{(-1)^{\ell(\lambda)}}{\prod_{j=1}^t m_j!\left(j!\right)^{m_j}}.
\]
\end{theorem}

\begin{TwoRemarks}~

\noindent
(1) A {\it weakly holomorphic modular form} is a meromorphic form whose poles (if any) are supported at cusps.
The modular form $f_F$ in Theorem~\ref{Theorem3} is explicitly given in Proposition~\ref{JacobiDecomp}.

\noindent
(2) One way to interpret Theorem~\ref{Theorem3} is that the generating function for traces of partition Eisenstein series for $\Div(F)=D$ is the explicit formula for $F(z;\tau)/f_F(\tau).$
\end{TwoRemarks}

This paper is organized as follows. In Section 2 we recall P\'olya's theory of cycle index polynomials for symmetric groups, the key tool for all of the results in this note. In Section 3 we prove Theorem~\ref{Theorem1} by applying these results to the Weierstrass $\sigma$-function. In Section 4 (resp. Section 5) we prove Theorem~\ref{Theorem2} (resp. Theorem~\ref{Theorem3}) with P\'olya's formulas after recalling essential preliminaries.

\section*{Acknowledgements}
\noindent The authors thank the referee, Kathrin Bringmann and Badri Pandey for comments that improved this paper.
 The third author thanks the Thomas Jefferson Fund and the NSF
(DMS-2002265 and DMS-2055118). The fourth author is grateful for the support of a Fulbright Nehru Postdoctoral Fellowship.

\section{P\'olya's cycle index polynomials}

The structure of traces of partition Eisenstein series arises from the classical theory of the symmetric group, and their connection to integer partitions. Namely, the key tool is P\'olya's theory of cycle index polynomials  (for example, see \cite{Stanley}). Recall that  a partition $\lambda=(\lambda_1,\dots,\lambda_{\ell(\lambda)})\vdash k$ or $(1^{m_1},\dots,k^{m_k})\vdash k$, labels a conjugacy class by cycle type. Moreover, the number of permutations in $\mathfrak{S}_k$ of cycle type $\lambda$ is
is $k!/z_{\lambda}$, where 
 $z_{\lambda}:=1^{m_1}\cdots k^{m_k}m_1!\cdots m_k!$. The {\it cycle index polynomial} for the symmetric group $\mathfrak{S}_k$ is given by
\begin{align} \label{CIF}
Z(\mathfrak{S}_k)&=\sum_{\lambda\vdash k}\frac1{z_{\lambda}}\prod_{j=1}^{\ell(\lambda)}x_{\lambda_j}
=\sum_{\lambda\vdash k}\prod_{j=1}^k\frac1{m_j!}\left(\frac{x_j}{j}\right)^{m_j}.
\end{align}
We require the following generating function for these polynomials in $k$-aspect.
\begin{lemma}[Example 5.2.10 of \cite{Stanley}]\label{PolyaGenFunction} As a power series in $y$, the generating function for the cycle index polynomials satisfies
$$\sum_{k\geq0}Z(\mathfrak{S}_k)\,y^k=\exp\left(\sum_{j\geq1}x_j\cdot \frac{y^j}j\right).$$
\end{lemma}

\begin{example}
 Here are the first few examples of P\'olya's cycle index polynomials: 
$$Z(\mathfrak{S}_1)=x_1, \qquad Z(\mathfrak{S}_2)=\frac1{2!}(x_1^2+x_2), \qquad Z(\mathfrak{S}_3)=\frac1{3!}(x_1^3+3x_1x_2+2x_3).
$$
\end{example}

\section{Proof of Theorem~\ref{Theorem1}}

Theorem~\ref{Theorem1} follows from an application of the theory of  P\'olya's cycle index generating functions (i.e. Lemma~\ref{PolyaGenFunction}) to  Weierstrass's $\sigma$-function (for example, see Section I.5 of \cite{Silverman}).
We begin by considering the function 
\[
\sigma_s(z;\tau) := z\prod_{w\in \Lambda_\tau(s)} \left(1-\frac{z^2}{w}\right).
\]
By the Weierstrass Factorization Theorem, this product converges as long as the sum 
\[
 \sum_{w\in \Lambda_\tau(s)}\frac{1}{w}
\]
converges absolutely, which it does for $s>0.$ Since we always have $s>0,$ we may expand this expression as a Taylor series 
\begin{equation}\label{TaylorSigma s}
\sigma_s(z;\tau) = \sum_{k=0}^\infty (-1)^ke_k(\Lambda_{\tau}(s))z^{2k+1}. 
\end{equation}
We compare this function with the Weierstrass $\sigma$-function 
\begin{eqnarray}\label{sigmafunction}
\sigma(z,\tau):=z\prod_{\substack{w\in \Lambda_\tau\\ \Im (w)>0 \text{ or }w>0 }} \left(1-\frac{z^2}{w^2}\right)\exp\left(\tfrac{z^2}{w^2}\right).
\end{eqnarray}
Note that this product is taken over a half-lattice rather than the full lattice. 
Introducing an $s$-parameter and limit, we find that 
\begin{eqnarray*}\label{SigmaSigmaS}
\sigma(z,\tau) &=&\lim_{s\to 0^+} z\prod_{\substack{w\in \Lambda_\tau\\ \Im (w)>0 \text{ or }w>0 }}  \left(1-\frac{z^2}{w^2|w|^{2s}}\right)\exp\left(\tfrac{z^2}{w^2|w|^{2s}}\right)\\
&=&\lim_{s\to 0^+} \sigma_s(z,\tau)\exp\left( (2\pi i  z)^2 \frac{G_2^{\star}(\tau)}{2}\right).
\end{eqnarray*}
Here we use ``Hecke's trick'' (for example, see p. 84 of \cite{BFOR})
\[
\lim_{s\to 0^+} \sum_{\substack{w\in \Lambda_\tau\\ \Im (w)>0 \text{ or }w>0 }}\frac{1}{w^2|w|^{2s}}  
\ = \ \frac12 \lim_{s\to 0^+}\sum_{w \in \Lambda_\tau(s)} \frac{1}{w} \ = \  (2\pi i)^2 \frac{G_2^{\star}(\tau)}{2}.
\]

The logarithmic derivative of the $\sigma$ function (with respect to $z$) has the Taylor expansion
\[
\frac{\sigma'(z;\tau)}{\sigma(z;\tau)} = \frac{1}{z} -\sum_{k\geq2} \frac{G_{2k}(\tau)(2\pi i)^{2k}}{(2k-1)!}z^{2k-1}
\]
(see Prop. I.5.1 of \cite{Silverman}, where we note a difference in notation with our $G_{2k}(\tau)$ being $\frac{(2k-1)!}{(2\pi i)^{2k}}G_{2k}(\Lambda_\tau)$).
This gives us the exponential expansion of $\sigma$ as
 \[
 \sigma(z;\tau) = z\cdot \exp\left( -\sum_{k\geq 2}\frac{G_{2k}(\tau)}{(2k)!}(2\pi i z)^{2k}\right),
 \] 
 and so we have 
  \begin{equation}\label{SigmaSLimit}
\lim_{s\to 0^+} \sigma_s(z;\tau) = z\cdot \exp\left(-\frac{G^{\star}_2(\tau)}{2}(2\pi i z)^2 -\sum_{k\geq 2}\frac{G_{2k}(\tau)}{(2k)!}(2\pi i z)^{2k}\right).
 \end{equation}
Equating this expression with (\ref{TaylorSigma s}), we see that the functions $e_k(\Lambda_{\tau}(0))$ are sums of products of the $G_{2k}$ and $G_2^{\star}$ Eisenstein series. Replacing $G_2^{\star}$ with $G_2$ yields quasimodular forms.  This proves (1) and (2), where the polynomial $F_k(X_1, X_2,\dots, X_k)$ arises from these expressions.

To prove (3), which requires the derivation of $\phi_{\Lambda},$ one simply applies
Lemma~\ref{PolyaGenFunction} to  the expression 
   \[
z\cdot \exp\left( -\sum_{k\geq 1}\frac{G_{2k}(\tau)}{(2k)!}\cdot(2\pi i z)^{2k}\right) =z\cdot\sum_{k\geq 1} (2\pi i z)^{2k}
\sum_{\lambda=(1^{m_1},\dots, k^{m_k})\vdash k} (-1)^{\ell(\lambda)} \prod_{j=1}^k \frac{1}{m_j!}\left(\frac{G_{2j}(\tau)}{(2j)!}\right)^{m_j},
 \]
 where $x_j=\frac{-G_{2j}(\tau)}{2(2j-1)!}$ and $y=(2\pi i\, z)^2.$

\section{Proof of Theorem~\ref{Theorem2}}
Here we prove Theorem~\ref{Theorem2}, which gives the generating function for the even crank moments in terms of traces of partition Eisenstein series.
\subsection{Two lemmas}
For each positive integer $k$, we consider the Lambert series (see (\ref{Sj}))
$$
	\mathbf{S}_{2k-1}(q)=\sum_{m\geq1}\frac{m^{2k-1}q^m}{1-q^m} 
	= \frac{B_{2k}}{4k} +\frac{1}{2}G_{2k}(\tau). 
$$
The second expression follows from (\ref{Eisenstein}).
We require the following fact which is explained in the proof of Theorem 1.2 of \cite{AOS2} (see equation (3.6) of \cite{AOS2}).
\begin{lemma}\label{identity1}
As a power series in $X$, we have that
\begin{equation*}
	\exp\left(-2\sum_{k\geq1}\frac{\mathbf{S}_{2k-1}(q)}{(2k)!} (-4X^2)^k\right)=
	\prod_{j\geq1} \left[1+\frac{4(\sin^2X)q^j}{(1-q^j)^2}\right].
\end{equation*}
\end{lemma}

\noindent
We also require the following convenient generating function for normalized Bernoulli numbers (for example, see equation (1.518.1) of \cite{GR}).
\begin{lemma} \label{sinc} 
	As a power series in $X$, we have that
	\begin{displaymath}
		\begin{split}
			\frac{\sin(X)}{X}=\exp\left(\sum_{k\geq1}\frac{(-4)^kB_{2k}}{(2k)(2k)!}\cdot X^{2k}\right).
		\end{split}
	\end{displaymath}
\end{lemma}
\subsection{Proof of Theorem~\ref{Theorem2}}
	The generating function for $M(m,n)$ (see \cite{AG, Garvan}) is given by 	
\begin{equation}\label{AGGenFunction}
\sum_{n\geq0}\sum_{m\in\mathbb{Z}} M(m,n) z^mq^n
	=\prod_{n\geq1}\frac{1-q^n}{(1-zq^n)(1-z^{-1}q^n)}.
\end{equation}
	After letting $z=e^{2iX},$ we pair up conjugate terms to obtain
	$$\sum_{n\geq0}\sum_{m\in\mathbb{Z}} M(m,n)q^n e^{2imX}
	=\prod_{n\geq1}\frac{(1-q^n)}{(1-2\cos(2X)q^n+q^{2n})}.$$
Thanks to the identity $-2\cos(2X)=-2+4\sin^2X,$  after taking real parts of $z$ we obtain
	\begin{align}\label{CM}
		\sum_{n\geq0}\sum_{m\in\mathbb{Z}} M(m,n)q^n \cos(2mX)
= \prod_{j\geq1}\frac{1}{(1-q^{j})\left[1+\frac{4(\sin^2X)q^j}{(1-q^j)^2}\right]}.
	\end{align}
	Combining Lemma \ref{identity1} and \eqref{CM}, we obtain
	\begin{align*}
	\sum_{n\geq0}\sum_{m\in\mathbb{Z}} M(m,n)q^n \cos(2mX) =\frac{1}{(q;q)_{\infty}}\cdot\exp\left(2\sum_{k\geq1}\frac{\mathbf{S}_{2k-1}(q)}{(2k)!} (-4X^2)^k\right).
\end{align*}
Thanks to \eqref{Sj} and Lemma \ref{sinc}, we obtain
\begin{align*}
\sum_{n\geq0}\sum_{m\in\mathbb{Z}} M(m,n)q^n \cos(2mX)&=\frac{1}{(q;q)_{\infty}}\cdot\exp\left(\sum_{k\geq1}\frac{B_{2k}}{(2k)(2k)!} (-4X^2)^k\right)\exp\left(\sum_{k\geq1}\frac{G_{2k}(\tau)}{(2k)!} (-4X^2)^k\right)\\
&=\frac{1}{(q;q)_{\infty}}\cdot\frac{\sin(X)}{X}\cdot\exp\left(\sum_{k\geq1}\frac{G_{2k}(\tau)}{(2k)!} (-4X^2)^k\right).
\end{align*}
	We recognize this last expression in the context of P\'olya's cycle index polynomials. Namely,
	Lemma~\ref{PolyaGenFunction} gives the identity
	$$\exp\left( \sum_{j\geq1}x_j\cdot \frac{y^j}j   \right)
	=\sum_{k\geq0}\left(\sum_{\lambda\vdash k}  
	\prod_{j=1}^k \frac1{m_j!} \left(\frac{x_j}{j}\right)^{m_j} 
	\right)y^k,
	$$
	which we apply with $x_j=\frac{G_{2j}(\tau)}{2(2j-1)!}$ and $y=-4X^2.$ Therefore, we find that
	\begin{align*}
		\sum_{n\geq0}\sum_{m\in\mathbb{Z}} M(m,n)q^n \cos(2mX) =\frac{1}{(q;q)_{\infty}}\cdot\frac{\sin(X)}{X}\cdot
		\sum_{k\geq0} \left(\sum_{\lambda\vdash k} \prod_{j=1}^k \frac{1}{m_j!}
		\left(\frac{G_{2j}(\tau)}{(2j)!}\right)^{m_j}\right)(-4X^2)^k.
	\end{align*}
	Using the Taylor expansion of $\cos(2mX)$, a simple change of variables (i.e. $2X\rightarrow X$) gives	
	\begin{equation}\label{CrankGenFunction}
		\sum_{k\geq0}(-1)^kC_{2k}(q)\cdot\frac{X^{2k}}{(2k)!}
		=\frac{2\sin(\frac{X}{2})}{(q;q)_{\infty}}\cdot
		\sum_{k\geq0} (-1)^k\Tr_k(\phi_{\color{black}c};\tau)\cdot X^{2k-1}.
	\end{equation}
	
	\subsection{Proof of Corollary~\ref{Cor1}}
	By inserting
		 the Taylor expansion of $\sin(X/2)$ in (\ref{CrankGenFunction}), a straightforward simplification gives
	\begin{align*}
		\sum_{k\geq0}(-1)^kC_{2k}(q)\cdot\frac{X^{2k}}{(2k)!}
		=\frac{1}{(q;q)_{\infty}}\cdot\sum_{k=0}^\infty (-1)^k\left(
		\sum_{n=0}^{k} \frac{\Tr_{k-n}(\phi_{\color{black}c};\tau)}{4^{n}\cdot(2n+1)!}\right)\cdot X^{2k}.
	\end{align*} 
	By comparing the coefficient of $X^{2k}$ on both sides and using the Pochhammer symbol to take into account
$(2k)!$, we obtain the claimed formulas for $C_{2k}(q).$

\section{Proof of Theorem~\ref{Theorem3}}

In this section, we prove Theorem~\ref{Theorem3} after recalling some essential preliminaries about Jacobi forms.

\subsection{Background on Jacobi forms}\label{SecJacobi}
The definition of a meromorphic Jacobi form is given in the introduction. Here we give further basic properties of Jacobi forms to facilitate the proof of Theorem~\ref{Theorem3}. Namely, we require a universal factorization theorem for meromorphic Jacobi forms with torsional divisor.

\subsubsection{A key Jacobi form factorization}
To this end, we begin by recalling the first example of a Jacobi form, the celebrated Jacobi theta function (see \cite{EichlerZagier, Zagier})
\[
\theta(z;\tau):=\sum_{n\in \Z}u^nq^{n^2/2},
\]
which is a Jacobi form for $\SL_2(\Z)$ of weight $1/2$ and index $1/2$.
We work instead with a slightly modified version of this function. Namely,  let $\eta(\tau):=q^{\frac1{24}}\prod_{n=1}^{\infty}(1-q^n),$ and then consider 
\begin{align}\label{ThetaDef}
    \Theta(z,\tau) \ &= \  (u^{1/2}-u^{-1/2})\prod_{n\geq 1}\frac{(1-uq^n)(1-u^{-1}q^n)}{(1-q^n)^2}
= \ \frac{1}{\eta(\tau)^3} \ \sum_{n\in \mathbb{Z}} (-1)^n u^{\frac{2n+1}{2}}q^{\frac{(2n+1)^2}{8}},
\end{align}
where we have employed the Jacobi Triple Product formula (see Theorem 2.8 of \cite{Andrews}). We note that this function arises naturally as 
a factor of the generating function for the Andrews-Garvan crank (\ref{AGGenFunction}), and
it is a Jacobi form for $\SL_2(\Z)$ with weight $-1$ and index $1/2$. 
The main reason for working with $\Theta(z;\tau)$ is its close connection to the Weierstrass $\sigma$-function.  To be precise, it satisfies the identity (see Theorem I.6.4 of \cite{Silverman}) 
\begin{equation}\label{ThetaSigma}
\Theta(z;\tau)= (2\pi i)\exp\left(-\frac{G_2(\tau)}{2}(2\pi iz)^2\right)\sigma(z;\tau).
\end{equation}

As noted earlier, the divisor of a non-zero Jacobi form $F(z;\tau)$ is the formal finite sum \[\Div(F)=\displaystyle\sum_{x\in \C/\Lambda_{\tau}}a_x \cdot (x),\] where each $a_x$ is the the order of $F$ at the point $x$.  The {\it degree} of a divisor is the  sum 
\[\deg(D) = \sum_{x\in \mathbb{C}/\Lambda_\tau} a_x.\]
Using this terminology, we have the following factorization result for meromorphic Jacobi forms. This proposition is the key point that identifies the weakly holomorphic modular form $f_F$ in Theorem~\ref{Theorem3}.

\begin{proposition}\label{JacobiDecomp}
Let $F(z;\tau)$ be a non-zero meromorphic Jacobi form of weight $k$ and index $m$ for some subgroup $\Gamma\subseteq \SL_2(\mathbb{Z})$. If $F$ has divisor $D=\Div(F)=\sum_{i=0}^N a_{x_i}\cdot (x_i)$, with $x_0=0$,
then the following are true.

\noindent
(1) The action of $\Gamma$ fixes the divisor $D$. Moreover, we have that $\deg(D)=2m$, and as a sum of complex numbers, we have
\begin{equation}\label{condition}
\sum_{i=0}^N a_{x_i}\cdot  x_i \equiv 0  \pmod{\Lambda_\tau}.
\end{equation}

\noindent
(2)   If representatives $x_i$ of the divisor are chosen so that the sum in (\ref{condition}) vanishes, then the Jacobi form $F$ factors as the product
\[
F(z;\tau)=f_F(\tau)\cdot \Theta(z;\tau)^{a_{x_0}}\prod_{i=1}^{N} \frac{\Theta(z-x_i;\tau)^{a_{x_i}}}{\Theta(-x_i;\tau)^{a_{x_i}}},
\]
where $f_F(\tau)$ is a meromorphic modular form on $\Gamma$ with weight $k+a_{x_0}$.
\end{proposition}
\begin{remark}
We stress that $a_{x_0}$ can of course be 0 in $\Div(F)$ in Proposition~\ref{JacobiDecomp}. We set $x_0=0$ as this point requires separate treatment in the proof below.
\end{remark}

\begin{proof}
The points of $\C/\Lambda_\tau$ are fixed under the action of $\Gamma$ unless, like torsion points,  they depend on $\tau$. The  modular transformation of $F$ implies that $\ord_x(F(\tau))=\ord_x\big(F(\frac{z}{c\tau+d},\frac{a\tau+b}{c\tau+d})\big)$ for all $x\in \C/\Lambda_\tau$ and $\left(\begin{smallmatrix}a&b\\c&d\end{smallmatrix}\right)\in \Gamma$, and so the action of $\Gamma$ must fix $D$.

Now consider the Jacobi form 
\[
F_1(z;\tau):=F(z;\tau)\cdot \Theta(z;\tau)^{-2m},
\]
which is a Jacobi form of index $0.$ In other words, it is an elliptic function.  
From \cite[Th. 2.2] {SilvermanI}, we have that the degree of the divisor of $F_1$ is $0,$ and also that 
\[\sum_{x\in \C/\Lambda_\tau} \ord_x(F_1) \cdot x \equiv 0\pmod{\Lambda_\tau},\]
as a sum of genuine complex numbers (as opposed to a formal sum of points).
Since the divisor of $\Theta(z;\tau)$ is just $(0)$, the proof of (1) is complete. 

 We now turn to the proof of (2).
To this end, we consider the function
\[
f_F(z;\tau)=F(z;\tau)\cdot \Theta(z;\tau)^{-a_{x_0}}\prod_{i>0} \frac{\Theta(-x_i;\tau)^{a_{x_i}}}{\Theta(z-x_i;\tau)^{a_{x_i}}}.
\]
The divisor of this function is trivial, and applying the elliptic transformation laws of  $\Theta(z;\tau)$, we see that it is bounded as $z$ varies. Thus $f_F$ is constant in $z$. A short calculation using the transformation laws of $\Theta(z;\tau)$ shows that the exponential terms from the Jacobi transformation laws cancel and so $f_F(\tau)$ is modular on $\Gamma$ with weight $k+a_{x_0}.$ 
\end{proof}

\subsubsection{Eisenstein series for torsional divisors}
Theorem~\ref{Theorem3} requires Eisenstein series that correspond to a torsional divisor $D$. These series arise naturally from $\Theta(z;\tau),$
and thanks to Proposition~\ref{JacobiDecomp} 
{\color{black} we are led} to Theorem~\ref{Theorem3}. 
For convenience, we introduce the minor modification
\begin{equation}\label{ThetaTilde}
\widetilde \Theta(z;\tau):= \exp\left(\tfrac{\pi}2\tfrac{z^2-|z|^2}{ \Im(\tau)} \right)\Theta(z;\tau).
\end{equation}
The Jacobi form transformation laws of $\Theta(z;\tau)$ (i.e. weight $-1$ and index $1/2$), imply that
\[\tau\widetilde\Theta\left(\frac{z}{\tau};\frac{-1}{\tau}\right)=\widetilde\Theta(z;\tau)=\widetilde\Theta(z;\tau+1)\qquad \text{and} \qquad \left|\widetilde\Theta(z+1;\tau)\right|=\left|\widetilde\Theta(z+\tau;\tau)\right|=\left|\widetilde\Theta(z;\tau)\right|.
\]
Therefore, the Taylor coefficients of $\log |\Theta(z;\tau)|$ and $\log |\widetilde \Theta(z;\tau)|$ are well-defined. 
The Eisenstein series that will form the building blocks of (\ref{highereisenstein}) arise in this way.

To be precise, if $x$ is a non-zero torsion point, we define the weight $k$ Eisenstein series $\widetilde G_{k,x}(\tau)$ using the Taylor coefficients of $\widetilde \Theta(z;\tau)$ around $z=x:$
\begin{equation}\label{TildeGkx}
\widetilde G_{k,x}(\tau) \ := \ \left.-\left(\frac{1}{2\pi i}\frac{d}{dz}\right)^k \log\left|\widetilde \Theta(z;\tau)\right|^2 \right|_{z=x}.
\end{equation}
Similarly, we define the holomorphic Eisenstein series
\begin{equation}\label{Gkx}
G_{k,x}(\tau) \ := \ \left.-\left(\frac{1}{2\pi i}\frac{d}{dz}\right)^k \log\left| \Theta(z;\tau)\right|^2 \right|_{z=x}.
\end{equation}

The following lemma lists the key properties that we require of these Eisenstein series.

\begin{lemma}\label{TorsionalEisenstein} If $x=\alpha \tau+\beta,$  where $\alpha,\beta \in \Q,$ then the following are true.

\noindent
(1) If  $x'$ is a torsion point with $x-x'\in \Lambda_{\tau},$ then 
$$\widetilde G_{k,x}(\tau)=\widetilde G_{k,x'}(\tau).
$$ 
Moreover, if $\gamma=\left(\begin{smallmatrix}a&b\\c&d\end{smallmatrix}\right)\in \SL_2(\Z)$, then 
\[\widetilde G_{k,x}(\gamma\tau) = (c\tau+d)^k \widetilde G_{k,x'}(\tau),
\] where $x' = \alpha'\tau+\beta'$ with $(\alpha',\beta')=(\alpha,\beta)\gamma.$

\noindent
(2) The series $\widetilde G_{k,x}(\tau)$ and $G_{k,x}(\tau)$ are related by
\[ 
\widetilde G_{k,x}(\tau)=\begin{cases} 
G_{k,x}(\tau) +\alpha &\ \ \ \text{ if } k=1,\\
G_{k,x}(\tau) +\frac{1}{4\pi \Im(\tau)}&\ \ \ \text{ if } k=2,\\
G_{k,x}(\tau) &\ \ \ \text{ if } k>2.
\end{cases}
\] 
 
 \noindent 
 (3)
 If $\alpha=0$, then $G_{k,x}(\tau)$ has the Fourier expansion
\[
G_{k,x}(\tau)=
\begin{cases} 
\tfrac{1}{2}+\frac{1}{2\pi i}\left(\zeta(1,-\beta) -\zeta(1,\beta)\right)+\displaystyle\sum_{n> 0} \sum_{m>0} \left(\zeta_\beta^n-\zeta_{\beta}^{-n}\right) q^{mn}
&\text{ if } k=1,\\
\frac{(k-1)!}{(2\pi i)^k}\left(\zeta(k,-\beta) +(-1)^k\zeta(k,\beta)\right) +\displaystyle\sum_{n> 0} \sum_{m>0} n^{k-1}\left(\zeta_\beta^n+(-1)^k\zeta_{\beta}^{-n}\right) q^{mn}
&\text{ if } k\geq 2,
\end{cases}
\]
where $\zeta_\beta:=e^{\frac{2\pi i}{\beta}},$  and $\zeta(s,\beta)$ is the Hurwitz zeta function.

\noindent
(4) If $0<\alpha<1$,  then $G_{k,x}(\tau)$ has the Fourier expansion
\[
G_{k,x}(\tau)=
\begin{cases} 
\ \frac{1}{2}+\displaystyle\sum_{n> 0} \sum_{m\geq 0} \zeta_{\beta}^n q^{(m+a)n}- \sum_{n>0} \sum_{m>0} \zeta_{\beta}^{-n} q^{(m-a)n}
&\text{ if } k=1,\\
\ \displaystyle\sum_{n> 0} \sum_{m\geq0} n^{k-1}\zeta_{\beta}^n q^{(m+a)n}- \sum_{n>0} \sum_{m>0} (-n)^{k-1}\zeta_{\beta}^{-n} q^{(m-a)n}
 &\text{ if } k\geq2.
\end{cases}
\]
 
\end{lemma}
\begin{TwoRemarks} ~

\noindent
(1) We note that Lemma~\ref{TorsionalEisenstein} (1) implies that $\widetilde{G}_{k,x}(\tau)$ is modular with level $\lcm(\denom(\alpha),\denom(\beta)).$ The same levels hold in all cases for $G_{k,x}(\tau).$

\noindent
(2) We emphasize the fact that $x-x'\in \Lambda_\tau$ does not imply that $G_{1,x}(\tau)$ equals $G_{1,x'}(\tau)$. In particular,  Lemma~\ref{TorsionalEisenstein} (1) and (2) show that $G_{1,x}(\tau)\neq G_{1,x+c\tau}(\tau)$ if $c\neq 0$. 
\end{TwoRemarks}

\begin{proof}
The transformation law in (1) follows from the transformation law for $\widetilde \Theta(z;\tau)$, and the fact that $$\alpha \frac{a\tau+b}{c\tau+d} +\beta= \frac{\alpha'\tau+\beta'}{c\tau+d}.
$$
To prove (2), we simply note that the additional terms for $\widetilde G_{k,x}(\tau)$ when $k=1$ and $k=2$ are simply the first and second logarithmic derivatives respectively of the exponential factor $\exp\left(\tfrac{\pi}2\tfrac{z^2-|z|^2}{ \Im (\tau)} \right)$ in the definition of $\widetilde \Theta(z;\tau)$ in (\ref{ThetaTilde}). 
 
 We take the logarithm of (\ref{ThetaDef}) to obtain
 \[
 \log|\Theta(z;\tau)|^2=\log|u^{-1/2}|^2+\log|1-u|^2+\sum_{n\geq 1} \left(\log|1-uq^n|^2+\log|1-u^{-1}q^n|^2 {\color{black} - \log\vert 1-q^n\vert^4 }\right).
 \] 
 Then applying the iterated derivative, and the substitution $u=\zeta_{\beta}q^{\alpha}$ leads to the Fourier expansions in (3) and (4).  The Hurwitz $\zeta$-function values in (3) are
 \[
 \left.-\left(\frac{1}{2\pi i}\frac{d}{dz}\right)^{k-1} \frac{u}{1-u}\right|_{z=\beta}= \frac{(k-1)!}{(2\pi i)^k}\left(\zeta(k,-\beta) +(-1)^k\zeta(k,\beta)\right).
 \]
 This completes the proof.
\end{proof}

For Theorem~\ref{Theorem3}, we must extend the Eisenstein series $G_{k,x}(\tau)$ to the weight $k$ quasimodular Eisenstein series $G_{k,D}(\tau)$ that we require.
To this end, we use the convention $G_{2k,0}(\tau)=G_{2k}(\tau)$ and $G_{2k+1,0}(\tau)=0$. For nontrivial torsional divisors  $D=\sum_{x\in \C/\Lambda_\tau} a_x(x)$,  we assume that representatives $x_i$ are chosen so that $|\sum_{x\in \C/\Lambda_\tau} a_x x|$ is  minimized. Then we let
\begin{equation}\label{EisensteinD}
G_{k,D}(\tau):=\sum_{x\in \C/\Lambda_\tau} a_x\, G_{k,x}(\tau).
\end{equation}

\begin{lemma}\label{EisensteinDStuff} If $D$ is a nontrivial torsional divisor for a Jacobi form $F$, then the following are true.

\noindent
(1) If $k=1$, then
\[
G_{1,D}(\tau)=\sum_{x\in \C/\Lambda_\tau} a_x\, \widetilde G_{1,x}(\tau)
\]
is a weight 1 holomorphic modular form.

\noindent
(2)  For $k=2,$ we have that $F$ has index $m=0$ if and only if
\[
G_{2,D}(\tau)=\sum_{x\in \C/\Lambda_\tau} a_x \,\widetilde G_{2,x}(\tau)
\]
is a weight 2 holomorphic modular form. Otherwise, $G_{2,D}(\tau)$ is quasimodular.

\noindent
(3) If $k\geq 3,$ then $G_{k,D}(\tau)$ is a weight $k$ holomorphic modular form.

\end{lemma}

\begin{proof}
In each case, we use the fact that $\widetilde G_{k,x}(\tau)$ is a weight $k$ (possibly non-holomorphic) modular form, and we then apply Lemma~\ref{TorsionalEisenstein} (2) to account for the difference
\[
G_{k,D}(\tau)-\sum_{x\in \C/\Lambda_\tau} a_x \widetilde G_{k,x}(\tau).
\]

In (1), the difference vanishes due to the requirement that 
$\sum_{x\in \C/\Lambda_\tau} a_x \cdot x \in \Lambda_\tau,$ and is therefore $0$ since we chose representatives $x$ to minimize this sum. For (2), we use the fact that
$\sum_{x\in \C/\Lambda_\tau} a_x =2m$, to see that the difference is $\frac{m}{2\pi \Im(\tau)}.$
For (3), this difference is identically $0$.
\end{proof}

\subsection{Proof of Theorem~\ref{Theorem3}}

Suppose that $F(z;\tau)$ is as in the statement of the Theorem \ref{Theorem3}. Moreover, suppose that $\Div(F)=\sum_{i=1}^Na_{x_i}(x_i),$ and again that $x_0=0$. 
Using the definition of  $G_{k,x}(\tau)$ above (i.e. (\ref{TildeGkx}) and (\ref{Gkx})), for $z$ near $0$ we find that
\begin{eqnarray*}
\Theta(z-x;\tau)&=&2\pi i(z-x)\exp\left( -\sum_{k\geq 1} \frac{G_{2k}(\tau)}{(2k)!}(2\pi i)^{2k}(z-x)^{2k}\right)\\
&=&-2\pi i\,x\cdot \exp\left( \log\left(1-\frac{z}{x}\right)-\sum_{k\geq 1} G_{2k}(\tau)(2\pi i)^{2k}\sum_{0\leq j\leq 2k} \frac{z^{j}}{j!}\frac{(-x)^{2k-j}}{(2k-j)!}\right)\\
&=&\Theta(-x;\tau)\cdot \exp\left( 
\sum_{j\geq 1}(2\pi i)^j \, \frac{z^{j}}{j!}\left(-\frac{(j-1)!}{(2\pi i \, x)^{j}} 
-\sum_{k\geq 1} G_{2k}(\tau) \frac{(-2\pi i x)^{2k-j}}{(2k-j)!}\right)\right)\\
&=&\Theta(-x;\tau)\cdot \exp\left( 
-\sum_{j\geq 1}\frac{(2\pi i  z)^{j}}{j!}G_{j,x}(\tau)\right).
\end{eqnarray*}
This implies then that
\begin{equation}\label{piece}
\frac{\Theta(z-x;\tau)}{\Theta(-x;\tau)}=\exp\left( 
-\sum_{j\geq 1}\frac{(2\pi i  z)^{j}}{j!}G_{j,x}(\tau)\right).
\end{equation}

We now recall Proposition~\ref{JacobiDecomp}, which gives 
\[
F(z;\tau)=f_F(\tau)\cdot \Theta(z;\tau)^{a_{x_0}}
\prod_{i=1}^{N} \frac{\Theta(z-x_i;\tau)^{a_{x_i}}}{\Theta(-x_i;\tau)^{a_{x_i}}} .
\]
Therefore, by taking into account the points of $D$, we find that (\ref{piece}) gives
$$
\log\frac{F(z;\tau)}{f_F(\tau)} = - \sum_{i=0}^N a_{x_i} \sum_{j\geq 1}\frac{(2\pi iz)^{j}}{j!}G_{j,x_i}(\tau)
=- \sum_{j\geq 1}\frac{(2\pi i z)^{j}}{j!}G_{j,D}(\tau),
$$
where in the last line we have used the definition of $G_{k,D}(\tau)$.
To complete the proof, we now apply Lemma~\ref{PolyaGenFunction} with 
$x_j= \frac{1}{(j-1)!}G_{j,D(\tau)}$ and  $y=2\pi i z$ to the exponential of this last expression.

\end{document}